\documentclass[11pt,a4paper]{article}

\usepackage{graphicx}
\usepackage[T1]{fontenc}
\usepackage{mathptmx}
\usepackage{amsmath,amssymb,amsthm}
\usepackage{bm}
\usepackage{indentfirst}
\usepackage{color}
\usepackage[colorlinks=true,linkcolor=blue,urlcolor=blue,citecolor=blue]{hyperref}
\usepackage{float}
\usepackage{tikz}
\usepackage{bbding}
\usepackage{subcaption}
\usepackage{enumitem}
\usepackage{booktabs}
\usepackage{algorithm}
\usepackage{algorithmic}
\usepackage{caption}
\DeclareCaptionLabelFormat{figlabel}{Fig#2}  % #2 代表编号，这里生成 Fig1, Fig2...
\usepackage{caption}
\usepackage{url}

\usetikzlibrary{shapes.geometric, calc, fit}
\allowdisplaybreaks[4]

\newtheorem{thm}{Theorem}
\newtheorem{lemma}{Lemma}[section]

\theoremstyle{definition}

\newtheorem{claim}{Claim}[section]
\newtheorem{remark}[thm]{Remark}

\newtheorem{problem}{Problem}

\begin{document}
	
	\title{Spectral and edge conditions for negative Hamiltonian cycles in unbalanced signed graphs}
	\author{Zihao Zha\thanks{Corresponding author. E-mail: \texttt{107552400678@stu.xju.edu.cn}.}\\
		{\footnotesize College of Mathematics and System Science, Xinjiang University, Urumqi 830046, China}}
	\date{}
	
	\maketitle
	
	{\flushleft\large\bf Abstract:}
	Spectral conditions for Hamiltonian cycles are a fundamental subject in extremal spectral graph theory. For simple unsigned graphs, Fiedler and Nikiforov [Linear Algebra Appl., 432 (2010), 2170--2173] pioneered this direction, and Li and Ning [Linear Multilinear Algebra, 64 (2016), 2252--2269] later extended these results by introducing minimum degree parameters. In this paper, we study the spectral conditions for the existence of a negative Hamiltonian cycle in unbalanced signed graphs of order $n > 8$. We first determine the maximum index among all \(n\)-vertex unbalanced signed graphs containing no negative Hamiltonian cycle, and characterize the corresponding extremal graphs up to switching isomorphism. Furthermore, we establish an upper bound on the number of edges in unbalanced signed graphs of order $n \ge 7$ with frustration index one, and classify all extremal graphs achieving this bound.
	
	\vspace{0.1cm}
	\begin{flushleft}
		\textbf{Keywords:} Signed graph; Largest eigenvalue; Negative Hamiltonian cycle; Edge extremal problem
	\end{flushleft}
	\textbf{AMS Classification:} 05C50; 05C22
	
	\section{Introduction}\label{1}
	
	Let \(G\) be a simple graph with vertex set \(V(G)\) and edge set \(E(G)\), and let \(e(G) = |E(G)|\). An underlying graph \(G\) together with a signature \(\sigma: E(G) \to \{-1, +1\}\) forms a signed graph \(\Gamma = (G, \sigma)\). In a signed graph, edge signs are usually interpreted as \(\pm 1\). An edge \(e\) is positive (resp. negative) if \(\sigma(e) = +1\) (resp. \(\sigma(e) = -1\)). A cycle in \(\Gamma\) is said to be positive if it contains an even number of negative edges; otherwise it is negative. A Hamiltonian cycle is called negative if the product of the signs of the edges in the cycle is negative. The signed graph \(\Gamma\) is balanced if it has no negative cycles; otherwise it is unbalanced. Let \(U \subset V(G)\). The operation of reversing the signs of all edges between \(U\) and \(V(G) \setminus U\) is called a switching operation. If a signed graph \(\Gamma'\) is obtained from \(\Gamma\) by applying finitely many switching operations, then \(\Gamma\) is said to be switching equivalent to \(\Gamma'\). Two signed graphs are called switching isomorphic if one of them is switching equivalent
	to an isomorphic copy of the other. For a signed graph \(\Gamma = (G, \sigma)\), we write \(e(\Gamma) = e(G)\). Let \(N_{\Gamma}(v_i)\) denote the set of neighbors of a vertex \(v_i\) in \(\Gamma\) and \(d_{\Gamma}(v_i) = \lvert N_{\Gamma}(v_i) \rvert\) the degree of \(v_i\) in \(\Gamma\). The adjacency matrix of \(\Gamma\) is defined as \(A(\Gamma) = (a_{ij}^\sigma)\), where \(a_{ij}^\sigma = \sigma(v_i v_j)\) if \(v_i \sim v_j\), and \(a_{ij}^\sigma = 0\) otherwise. The eigenvalues of \(\Gamma\) are the eigenvalues of \(A(\Gamma)\), and \(\lambda_1(\Gamma)\) denotes the index of \(\Gamma\). 
	
	\noindent Let \[
	\lambda_1(\Gamma) \ge \lambda_2(\Gamma) \ge \cdots \ge \lambda_n(\Gamma)
	\] be the eigenvalues of  $A(\Gamma)$. The spectral radius of $\Gamma$ is 
	\[
	\rho(\Gamma) = \max_{1 \leq i \leq n} |\lambda_i(\Gamma)|=\max\{\lambda_1(\Gamma), -\lambda_n(\Gamma)\}.
	\]
	
	 \noindent For more details about the notion of signed graphs, we refer to \cite{ Zaslavsky1982}.
	
	Let $\mathcal{F}$ be a family of graphs. A graph $G$ is $\mathcal{F}$-free if $G$ does not contain any graph in $\mathcal{F}$ as a subgraph. The classical spectral Tur\'an-type problem is to determine the maximum spectral radius of an $\mathcal{F}$-free graph of order $n$. This problem was originally proposed by Nikiforov \cite{Nikiforov2010}. For more information on the unsigned version, we refer the readers to \cite{Nikiforov2007, Nikiforov2008, Wilf1986}. In this paper, we focus on the spectral Tur\'{a}n-type problems in signed graphs.
	
		Denote by $C_r^{-}$ and $K_r^{-}$ the collections of all unbalanced signed graphs whose underlying graphs are $C_r$ and $K_r$, respectively.  
	Recently, the spectral radius or index of signed graphs forbidding certain negative subgraphs has been extensively investigated.  
	In 2024, Wang, Hou and Li~\cite{Wang2024} established the maximum spectral radius for $C_3^{-}$-free signed graphs.  
	In the same year, Wang and Lin~\cite{wang2023} obtained the maximum index among $C_4^{-}$-free signed graphs.  
	In 2025, Wang, Hou and Huang~\cite{WangHouHuang2025} determined the largest index of $C_{2k+1}^{-}$-free signed graphs for $3 \le k \le \frac{n}{10}-1$.  
	Concerning $K_r^{-}$-free signed graphs, Chen and Yuan~\cite{chen2023} in 2024 solved the problem of the maximum spectral radius for $K_4^{-}$-free graphs, while Wang~\cite{Wang2024Spectral} addressed the case of $K_5^{-}$ in the same year.  
	Additionally, Xiong and Hou~\cite{Xiong2024} characterized the extremal index of $K_{r+1}^{-}$-free signed graphs for $3 \le r \le n-1$. Very recently, Li and Qin~ \cite{li2026} determined the maximum index among all \(tC_3^-\)-free unbalanced signed graphs and further characterized the second maximum index.
	
 A cycle is Hamiltonian if it visits every vertex of the graph exactly once. A graph containing a Hamiltonian cycle is called a Hamiltonian graph. Over the past decade, considerable effort has been devoted to exploring spectral conditions for the Hamiltonicity of unsigned graphs. In 2010, Fiedler and Nikiforov \cite{FiedlerNikiforov2010} initiated this direction by establishing foundational spectral thresholds for graphs containing a Hamiltonian cycle. Subsequently, Lu, Liu, and Tian \cite{LuLiuTian2012} provided further spectral characterizations for Hamiltonian graphs. Significant progress was later achieved by introducing minimum degree constraints; specifically, Ning and Ge \cite{NingGe2015} derived refined conditions for graphs with minimum degree at least two, and Benediktovich \cite{Benediktovich} further sharpened the spectral threshold to \(n-3\) for such graphs. Later, by treating the minimum degree as a variable parameter, Li and Ning \cite{LiNing2016} provided a unified extension of the earlier spectral results \cite{FiedlerNikiforov2010, NingGe2015}. 
 
 However, the Hamiltonicity results discussed in the preceding paragraph concern unsigned graphs. A natural question arises: what spectral conditions are there for the existence of a negative Hamiltonian cycle in unbalanced signed graphs? This amounts to the following spectral Tur\'{a}n-type problem for signed graphs.
 
 \begin{problem}\label{prob:1}
 	What is the maximum index among all  unbalanced signed graphs of order $n > 8$ that contain no negative Hamiltonian cycle as a subgraph?
 \end{problem}
 
 In this paper, we completely resolve the above problem. Let \((K_{n-1}, \sigma)\) be a signed complete graph on the vertex set \(\{v_1, v_2, v_4, \ldots, v_n\}\) with exactly one negative edge. The signed graph $\Gamma_1$ is obtained by appending a new vertex $v_3$ to $(K_{n-1},\sigma)$ and joining $v_3$ to the two vertices $v_1, v_2$, with $v_1v_2$ being the unique negative edge (shown in red); all other edges are taken to be positive. The triangle \(v_1v_2v_3v_1\) is negative, and hence \(\Gamma_1\) is unbalanced. Moreover, \(d_{\Gamma_1}(v_3) = 2\), so every Hamiltonian cycle of the underlying graph must contain both \(v_1v_3\) and \(v_2v_3\). Such a cycle cannot also contain \(v_1v_2\), since otherwise the three edges \(v_1v_2, v_2v_3, v_3v_1\) would form a proper cycle on \(\{v_1, v_2, v_3\}\), leaving the remaining vertices outside the Hamiltonian cycle. Thus every Hamiltonian cycle avoids the unique negative edge \(v_1v_2\), and therefore \(\Gamma_1\) contains no negative Hamiltonian cycle. Our main result is stated as follows.
 
 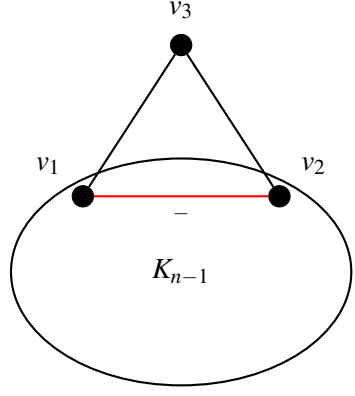
\begin{figure}[ht]
 	\centering
 	\begin{tikzpicture}[
 		vertex/.style={circle, draw, fill=black, minimum size=8pt, inner sep=0pt},
 		bigellipse/.style={ellipse, draw, thick, minimum width=4.5cm, minimum height=3cm},
 		blackedge/.style={draw=black, thick}
 		]
 		% 椭圆 K_{n-1}，中心在 (0,0)
 		\node[bigellipse] (K) at (0,0) {};
 		\node at (K.center) {\textsf{\(K_{n-1}\)}};
 		
 		% v1, v2 位于第二、第一象限，都在椭圆内部
 		\node[vertex, label={[label distance=2pt]above left:$v_1$}] (v1) at (-1.3, 1.0) {};
 		\node[vertex, label={[label distance=2pt]above right:$v_2$}] (v2) at ( 1.3, 1.0) {};
 		
 		% v3 在 y 轴正半轴，且在椭圆外面
 		\node[vertex, label={[label distance=2pt]above:$v_3$}] (v3) at ( 0, 3.0) {};
 		
 		% 连接 v1-v3, v2-v3 黑边
 		\draw[blackedge] (v1) -- (v3);
 		\draw[blackedge] (v2) -- (v3);
 		
 		% 连接 v1-v2 红边
 		\draw[red, thick] (v1) -- (v2) node[midway, below, text=black] {\scriptsize \(-\)};
 	\end{tikzpicture}
 	\caption{The signed graph \(\Gamma_1\) (red lines indicate negative edges).}
 	\label{fig:gamma1}
 \end{figure}
 
 \begin{thm}\label{thm:spectral}
 	Let $\Gamma = (G, \sigma)$ be an unbalanced signed graph of order $n > 8$. 
 	If $\lambda_1(\Gamma) \ge \lambda_1(\Gamma_1)$, then $\Gamma$ contains a negative Hamiltonian cycle unless $\Gamma$ is switching isomorphic to $\Gamma_1$ (see Fig.~\ref{fig:gamma1}).
 \end{thm}

 The study of edge conditions for Hamiltonian cycles in unsigned graphs is classical, initiated by Dirac's minimum degree theorem~\cite{Dirac1952} and followed by Ore's edge threshold~\cite{Ore1961}, Chv\'{a}tal's degree sequence condition~\cite{Chvatal1972}, and Bondy and Chv\'{a}tal's closure operation~\cite{BondyChvatal1976}. A major breakthrough came from Erd\H{o}s~\cite{Erdos1962}, who characterized the unique extremal non-Hamiltonian graph under a prescribed minimum degree. Later, F\"{u}redi, Kostochka, and Luo refined this result via stability analysis~\cite{FurediErdos2018} and extended it to Hamilton-connected graphs. Recently, Ai, Lei, Ning, and Shi~\cite{Ai2025} proposed a unified framework using feasible graph parameters. Building upon this, Liu, Ning, and Wang ~\cite{LiuNingWang2026} provided a complete extension of Erd\H{o}s's theorem, fully resolving the long-standing spectral and signless Laplacian spectral problems proposed by Li and Ning~\cite{LiNing2016}. It is then natural to ask what edge conditions guarantee the
 existence of a negative Hamiltonian cycle in an unbalanced signed graph. This amounts to the following Tur\'{a}n-type problem for signed graphs.
 
 \begin{problem}\label{prob:2}
 	What is the maximum number of edges in an unbalanced signed graph of order $n$ that contains no negative Hamiltonian cycle as a subgraph?
 \end{problem}
 
 The frustration index of a signed graph \(\Gamma\) is the minimum number of edges in \(\Gamma\) whose removal results in a balanced signed graph. In this paper, we address Problem~\ref{prob:2} under the additional assumption that the frustration index is 1. Our result is as follows.
 
 \begin{thm}\label{thm:edge}
 	Let $\Gamma = (G, \sigma)$ be an unbalanced signed graph of order $n \ge 7$ with frustration index $1$. 
 	If $\Gamma$ contains no negative Hamiltonian cycle as a subgraph, then
 	$e(\Gamma) \le \frac{n(n-1)}{2} - (n-3)$,
 	with equality if and only if $\Gamma$ is switching isomorphic to $\Gamma_1$ (see Fig.~\ref{fig:gamma1}).
 \end{thm}
 
	The paper is organized as follows. Section~\ref{1} introduces the main definitions and notation and states the main results. Section~\ref{2} is devoted to the proof of our main spectral result (Theorem~\ref{thm:spectral}). In Section~\ref{3}, we prove an edge extremal result (Theorem~\ref{thm:edge}).

% ==================== 第二个图：新图 + 原第二个图 ====================
\begin{figure}[ht]
	\centering
	
	% 左侧子图：新生成的图（带 v4 的 Γ1）
	\begin{subfigure}[b]{0.45\textwidth}
		\centering
		\begin{tikzpicture}[
			vertex/.style={circle, draw, fill=black, minimum size=8pt, inner sep=0pt},
			bigellipse/.style={ellipse, draw, thick, minimum width=4.5cm, minimum height=3cm},
			blackedge/.style={draw=black, thick}
			]
			\node[bigellipse] (K) at (0,0) {};
			\node at (K.center) {\textsf{\(K_{n-2}\)}};
			
			\node[vertex, label={[label distance=2pt]above left:$v_1$}] (v1) at (-1.3, 1.0) {};
			\node[vertex, label={[label distance=2pt]above right:$v_2$}] (v2) at ( 1.3, 1.0) {};
			
			% v3 向左移动，v4 放置在对称的右侧
			\node[vertex, label={[label distance=2pt]above:$v_3$}] (v3) at (-1.5, 3.0) {};
			\node[vertex, label={[label distance=2pt]above:$v_4$}] (v4) at ( 1.5, 3.0) {};
			
			% 连接黑边
			\draw[blackedge] (v1) -- (v3);
			\draw[blackedge] (v2) -- (v3);
			\draw[blackedge] (v1) -- (v4);
			\draw[blackedge] (v2) -- (v4);
			\draw[blackedge] (v1) -- (v2);

		\end{tikzpicture}
		\caption*{\(K_2 \vee (K_{n-4} \cup I_2)\)} % 如果你有新命名，可以修改这里
	\end{subfigure}
	\hfill
	% 右侧子图：原来的第二个图
	\begin{subfigure}[b]{0.45\textwidth}
		\centering
		\begin{tikzpicture}[
			vertex/.style={circle, draw, fill=black, minimum size=8pt, inner sep=0pt},
			bigellipse/.style={ellipse, draw, thick, minimum width=4.2cm, minimum height=2.8cm},
			blackedge/.style={draw=black, thick}
			]
			\node[bigellipse] (K) at (0,0) {};
			\node at (K.center) {\textsf{\(K_{n-2}\)}};
			
			\node[vertex, label={[label distance=2pt]above:$v_3$}] (v3) at ( 0, 1.5) {};
			\node[vertex, label={[label distance=2pt]above left:$v_1$}] (v1) at (-1.6, 3.0) {};
			\node[vertex, label={[label distance=2pt]above right:$v_2$}] (v2) at ( 1.6, 3.0) {};
			
			\draw[blackedge] (v1) -- (v2);
			\draw[blackedge] (v3) -- (v1);
			\draw[blackedge] (v3) -- (v2);
			
			\draw[blackedge] (v3) -- (K.north);
		\end{tikzpicture}
		\caption*{\(K_1 \vee (K_{n-3} \cup K_2)\)}
	\end{subfigure}
	
	\caption{The graphs \(K_2 \vee (K_{n-4} \cup I_2)\) and  \(K_1 \vee (K_{n-3} \cup K_2)\).}
	\label{fig:new_and_k1}
\end{figure}
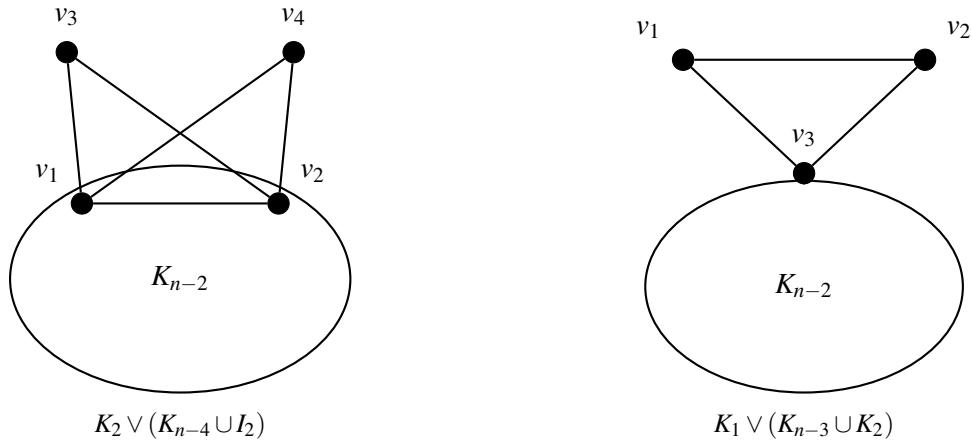
	
	\section{Proof of Theorem \ref{thm:spectral}}\label{2}
	
	Let $M$ be a real symmetric matrix of order $n$, and let $[n] = \{1, 2, \dots, n\}$. 
	Given a partition $\Pi: [n] = X_1 \cup X_2 \cup \cdots \cup X_k$, the matrix $M$ 
	can be written as
	\[
	M = \begin{bmatrix}
		M_{1,1} & M_{1,2} & \cdots & M_{1,k} \\
		M_{2,1} & M_{2,2} & \cdots & M_{2,k} \\
		\vdots & \vdots & \ddots & \vdots \\
		M_{k,1} & M_{k,2} & \cdots & M_{k,k}
	\end{bmatrix}.
	\]
	If all row sums of $M_{i,j}$ are the same, say $b_{i,j}$, for all $i, j \in \{1, 2, \dots, k\}$, 
	then $\Pi$ is called an \emph{equitable partition} of $M$, and the matrix $Q = (b_{i,j})_{k \times k}$ 
	is called an \emph{equitable quotient matrix} of $M$.
	
	\begin{lemma}[\cite{Brouwer2011}]\label{lemma:1}
		Let $M$ be a real symmetric matrix, and let $Q$ be an equitable quotient matrix of $M$. Then the matrix $M$ has the following two kinds of eigenvalues.
		\begin{enumerate}[label=(\roman*)]
			\item The eigenvalues coincide with the eigenvalues of $Q$.
			\item The eigenvalues of $M$ not in the spectra of $Q$ remain unchanged if some scalar multiple of the all-one block $J$ is added to block $M_{ij}$ for each $1 \le i, j \le k$.
		\end{enumerate}
	\end{lemma}
	
	We denote by $\delta(G)$ and $A(G)$ the minimum degree and the adjacency matrix of a graph $G$, respectively. The spectral radius $\lambda_1(G)$ is the largest eigenvalue of $A(G)$. The disjoint union $G_1\cup G_2$ has vertex set $V(G_1)\cup V(G_2)$ and edge set $E(G_1)\cup E(G_2)$. By starting with a disjoint union of two graphs \(G_1\) and \(G_2\) and adding edges joining every vertex of \(G_1\) to every vertex of \(G_2\), one obtains the join of \(G_1\) and \(G_2\), denoted \(G_1 \vee G_2\). Let $K_{s,t}$ be the complete bipartite graph with partitions of sizes $s$ and $t$. We write $K_n$ and $I_n$ for the complete graph and empty graph on $n$ vertices, respectively. 
	
	\begin{lemma}[\cite{Benediktovich}]\label{lemma:2}
		Let $G$ be a graph on $n > 8$ vertices with $\delta(G) \geq 2$. If
		\[
		\lambda_1(G) \geq n-3,
		\]
		then $G$ has a Hamiltonian cycle unless $G = K_2 \vee (K_{n-4} \cup I_2)$ or 
		$G = K_1 \vee (K_{n-3} \cup K_2)$. When $n=9$, there are two additional exceptions: 
		$G \in \{ I_5 \vee K_4,\ K_3 \vee (K_{1,4} \cup K_1) \}$.
	\end{lemma}
	
	Let $J$, $I$ and $O$ denote the all-ones matrix, the identity matrix and the all-zeros matrix, respectively.
	
	\begin{lemma}\label{lemma:3}
		Let $n > 8$ and \(\Gamma_1\) be the signed graph shown in Fig.~\ref{fig:gamma1}, and let \(K_2 \vee (K_{n-4} \cup I_2)\) and \(K_1 \vee (K_{n-3} \cup K_2)\) be the graphs shown in Fig.~\ref{fig:new_and_k1}. Then
		
		\begin{enumerate}
			\item[(i)] $\lambda_1(\Gamma_1) $ is the largest real root of 
			the equation $x^3 + (5-n)x^2 + (8-3n)x + 2n-8 =0 $. \label{lemma:3i}
			\item[(ii)] $\lambda_1(\Gamma_1) > n-\frac{5}{2} > \lambda_1(K_1 \vee (K_{n-3} \cup K_2))$ .\label{lemma:3ii}
			\item[(iii)] $\lambda_1(\Gamma_1) > n-\frac{5}{2} > \lambda_1(K_2 \vee (K_{n-4} \cup I_2))$ .\label{lemma:3iii}
		\end{enumerate}
		
	\end{lemma}
	
	\begin{proof}
		We first prove part (i). 
		By choosing a suitable equitable partition of $V(\Gamma_1)$, the adjacency matrix of $\Gamma_1$ can be written as
		\[
		A(\Gamma_1) =
		\begin{bmatrix}
			0 & -1 & 1 & J \\
			-1 & 0 & 1 & J \\
			1 & 1 & 0 & \bm{0} \\
			J & J & \bm{0} & J-I
		\end{bmatrix}
		\begin{array}{l}
			v_1 \\
			v_2 \\
			v_3 \\
			V(\Gamma_1)\setminus\{v_1,v_2,v_3\}
		\end{array}
		\]
		and its equitable quotient matrix is
		\[
		Q_1 =
		\begin{bmatrix}
			0 & -1 & 1 & n-3 \\
			-1 & 0 & 1 & n-3 \\
			1 & 1 & 0 & 0 \\
			1 & 1 & 0 & n-4
		\end{bmatrix}.
		\]
		By calculation, the characteristic polynomial of $Q_1$ is
		\[
		P_{Q_1}(x) = (x-1)\bigl(x^3 + (5-n)x^2 + (8-3n)x + 2n-8\bigr).
		\]
		\noindent It is easy to verify that $P_{Q_1}(n-\frac{5}{2})= -\frac{(2n - 11)(2n - 9)(2n - 7)}{16}  < 0$ and $P_{Q_1}(n-2)=4(n - 3)^2 > 0$.
		By Lemma~\ref{lemma:1}, all eigenvalues of $Q_1$ belong to the spectrum of $A(\Gamma_1)$ and by subtracting a suitable scalar multiple of the all-ones block from each subblock of $A(\Gamma_1)$, we derive the simplified block matrix
		\[
		A_1 =
		\begin{bmatrix}
			0 & 0 & 0 & \bm{0} \\    
			0 & 0 & 0 & \bm{0} \\    
			0 & 0 & 0 & \bm{0} \\  
			\bm{0} & \bm{0} & \bm{0} & -I_{n-3}
		\end{bmatrix}.
		\]
		
		\noindent  Since $\lambda_1(\Gamma_1) > 0 $, $\operatorname{Spec}(A_1) = \{-1^{[n-3]}, 0^{[3]}\}$ and $\lambda_1(A(\Gamma_1)) = \lambda_1(Q_1)$. Hence $\lambda_1(\Gamma_1) $ is the largest real root of the equation $x^3 + (5-n)x^2 + (8-3n)x + 2n-8=0 $. We get $\lambda_1(Q_1)> n-\frac{5}{2}$.
		
		We now prove part (ii). Note that
		\[
		A(K_1 \vee (K_{n-3} \cup K_2)) =
		\begin{bmatrix}
			0 & 1 & 1 & \bm{0} \\
			1 & 0 & 1 & \bm{0} \\
			1 & 1 & 0 & J \\
			\bm{0} & \bm{0} & J & J-I
		\end{bmatrix}
		\begin{array}{l}
			v_1 \\
			v_2 \\
			v_3 \\
			V(K_1 \vee (K_{n-3} \cup K_2))\setminus\{v_1,v_2,v_3\}
		\end{array}
		\]
		and its equitable quotient matrix is
		\[
		Q_2 =
		\begin{bmatrix}
			0 & 1 & 1 & 0 \\
			1 & 0 & 1 & 0 \\
			1 & 1 & 0 & n-3 \\
			0 & 0 & 1 & n-4
		\end{bmatrix}.
		\]
		By calculation, the characteristic polynomial of $Q_2$ is
		\[
		P_{Q_2}(x) = x^4 + (4-n)x^3 -nx^2 + (3n-14)x + 3n-11.
		\]
		
		\noindent It is easy to verify that $P_{Q_2}(n-\frac{5}{2}) =\frac{(2n - 3)(4n^2 - 20n - 3)}{16}> 0$ and $P_{Q_2}(n-3)=-2(n - 2) < 0$. By Lemma~\ref{lemma:1},  all eigenvalues of $Q_2$ belong to the spectrum of $A(K_1 \vee (K_{n-3} \cup K_2))$ and by subtracting a suitable scalar multiple of the all-ones block from each subblock of $A(K_1 \vee (K_{n-3} \cup K_2))$, we derive the simplified block matrix
		\[
		A_2 =
		\begin{bmatrix}
			0 & 0 & 0 & \bm{0} \\    
			0 & 0 & 0 & \bm{0} \\    
			0 & 0 & 0 & \bm{0} \\  
			\bm{0} & \bm{0} & \bm{0} & -I_{n-3}
		\end{bmatrix}.
		\]
		
		\noindent  Since $\lambda_1(K_1 \vee (K_{n-3} \cup K_2)) > 0 $, $\operatorname{Spec}(A_2) = \{-1^{[n-3]}, 0^{[3]}\}$ and $\lambda_1(A(K_1 \vee (K_{n-3} \cup K_2))) = \lambda_1(Q_2)$. Next, we show that \(P'_{Q_2}(x)>0\) for all \(x\ge n-\frac52\). The derivatives are
		\begin{align*}
			P'_{Q_2}(x)&=4x^3+(12-3n)x^2-2nx+3n-14,\\
			P''_{Q_2}(x)&=12x^2+(24-6n)x-2n,\\
			P'''_{Q_2}(x)&=24x+24-6n.
		\end{align*}
		For \(x\ge n-\frac52\) and $n > 8$,
		\[
		P'''_{Q_2}(x)\ge 24\left(n-\frac52\right)+24-6n=18(n-2)>0.
		\]
		Thus \(P''_{Q_2}(x)\) is strictly increasing for all \(x\ge n-\frac52\). Moreover,
		\[
		P''_{Q_2}\left(n-\frac52\right)=6n^2-23n+15>0\quad (n > 8),
		\]
		so \(P''_{Q_2}(x)>0\) for all \(x\ge n-\frac52\). Hence \(P'_{Q_2}(x)\) is strictly increasing for all \(x\ge n-\frac52\). Since
		\[
		P'_{Q_2}\left(n-\frac52\right)=\left(n-\frac52\right)^3+\frac52\left(n-\frac52\right)^2-2\left(n-\frac52\right)-\frac{13}{2}>0,
		\]
		we conclude that \(P'_{Q_2}(x)>0\) for all \(x\ge n-\frac52\). Therefore, \(P_{Q_2}(x)\) is strictly increasing for all \(x\ge n-\frac52\), and since \(P_{Q_2}\left(n-\frac52\right)>0\), it follows that
		\[
		P_{Q_2}(x)>0\quad\text{for all }x\ge n-\frac52.
		\] Thus, $\lambda_1(K_1 \vee (K_{n-3} \cup K_2)) = \lambda_1(Q_2)<n-\frac{5}{2}$. 
		 Therefore, $\lambda_1(\Gamma_1) > n-\frac{5}{2} > \lambda_1(K_1 \vee (K_{n-3} \cup K_2))$.
		 
		Finally, we prove part (iii). Note that
		\[
		A(K_2 \vee (K_{n-4} \cup I_2)) =
		\begin{bmatrix}
		
			\bm{0} & J & \bm{0}  \\
			J & J-I & J  \\
			\bm{0} & J & J-I
		\end{bmatrix}
		\begin{array}{l}
			\{v_3,v_4\} \\
			\{v_1,v_2\} \\
			V(K_2 \vee (K_{n-4} \cup I_2))\setminus\{v_1,v_2,v_3,v_4\}
		\end{array}
		\]
		and its equitable quotient matrix is
		\[
		Q_3 =
		\begin{bmatrix}
			0 & 2  & 0 \\
			2 & 1 & n-4  \\
			0 & 2 & n-5 \\
		\end{bmatrix}.
		\]
		By calculation, the characteristic polynomial of $Q_3$ is
		\[
		P_{Q_3}(x) = x^3 + (4-n)x^2 + (-n-1)x + 4n - 20.
		\]
		
		\noindent It is easy to verify that $P_{Q_3}(n-\frac{5}{2})= \frac{(2n - 13)(2n + 5)}{8} > 0$ and $P_{Q_3}(n-3) = -8 < 0$. By Lemma~\ref{lemma:1},  all eigenvalues of $Q_3$ belong to the spectrum of $A(K_2 \vee (K_{n-4} \cup I_2))$ and by subtracting a suitable scalar multiple of the all-ones block from each subblock of $A(K_2 \vee (K_{n-4} \cup I_2))$, we derive the simplified block matrix
		\[
		A_3 =
		\begin{bmatrix}
			\bm{0} & \bm{0} & \bm{0} \\    
			\bm{0} & -I_2 & \bm{0} \\      
			\bm{0} & \bm{0} & -I_{n-4}
		\end{bmatrix}.
		\]
		\noindent  Since $\lambda_1(K_2 \vee (K_{n-4} \cup I_2)) > 0 $, $\operatorname{Spec}(A_3) = \{-1^{[n-2]}, 0^{[2]}\}$ and $\lambda_1(A(K_2 \vee (K_{n-4} \cup I_2))) = \lambda_1(Q_3)$. Next, we show that \(P'_{Q_3}(x)>0\) for all \(x\ge n-\frac52\). The derivative is
		\[
		P'_{Q_3}(x) = 3x^2 + 2(4-n)x + (-n-1) = 3x^2 + (8-2n)x - n - 1.
		\]
		This is a quadratic function with a positive leading coefficient, so it opens upward. Its axis of symmetry is
		\[
		x_0 = \frac{-(8-2n)}{2\cdot 3} = \frac{n-4}{3}.
		\]
		Since $n > 8$, we have
		\[
		\frac{n-4}{3} \le n - \frac52.
		\]
		Therefore the axis of symmetry lies to the left of (or at) the left endpoint of the interval $\left[n-\frac52, \infty\right)$. Hence $P'_{Q_3}(x)$ is increasing on this interval, and its minimum value there occurs at $x = n - \frac52$.
		
		\noindent Evaluating the derivative at $x = n - \frac52$, we get
		\begin{align*}
			P'_{Q_3}\left(n-\frac52\right)
			&= 3\left(n-\frac52\right)^2 + (8-2n)\left(n-\frac52\right) - n - 1 \\
			&= n^2 - 3n - \frac94 \\
			&= (n-4)(n+1) + \frac74.
		\end{align*}
		For $n > 8$, we have $(n-4)(n+1) \ge 4 \cdot 9 = 36 > 0$, so
		\[
		P'_{Q_3}\left(n-\frac52\right) > 0.
		\]
		Since $P'_{Q_3}(x)$ is increasing on $\left[n-\frac52, \infty\right)$ and its value at the left endpoint is positive, it follows that
		\[
		P'_{Q_3}(x) \ge P'_{Q_3}\left(n-\frac52\right) > 0
		\]
		for all $x \ge n - \frac52$. Therefore, \(P_{Q_3}(x)\) is strictly increasing for all \(x\ge n-\frac52\), and since \(P_{Q_3}\left(n-\frac52\right)>0\), it follows that
		\[
		P_{Q_3}(x)>0\quad\text{for all }x\ge n-\frac52.
		\] Thus, $\lambda_1(K_2 \vee (K_{n-4} \cup I_2)) = \lambda_1(Q_3)<n-\frac{5}{2}$. 
		Therefore, $\lambda_1(\Gamma_1) > n-\frac{5}{2} > \lambda_1(K_2 \vee (K_{n-4} \cup I_2))$.
	\end{proof}

	\begin{lemma}[\cite{Stanic2018}]\label{lemma:4}
		Let $\Gamma$ be a signed graph. Then there exists a signed graph $\Gamma'$ switching 
		equivalent to $\Gamma$ such that $A(\Gamma')$ has a non-negative eigenvector 
		corresponding to $\lambda_1(\Gamma')$.
	\end{lemma}
	
	\begin{lemma}[\cite{Zaslavsky1982}]\label{lemma:5}
		Two signed graphs with the same underlying graph are switching equivalent if and 
		only if they have the same set of positive cycles.
	\end{lemma}
	
	\begin{lemma}[\cite{Xiong2024}]\label{lemma:6}
		Let $\Gamma = (G,\sigma)$ be a signed graph with a non-negative unit eigenvector 
		$X = (x_1, x_2, \dots, x_n)^{\top}$ corresponding to the largest 
		eigenvalue $\lambda_1(\Gamma)$. If we perform one of the
		following perturbations in $\Gamma $.
		\begin{enumerate}[label=(\roman*)]
			\item  Adding some positive edges\label{lemma:6i},
			\item Removing some negative edges, \label{lemma:6ii}
			\item Reversing the signs of some negative edges\label{lemma:6iii},
		\end{enumerate}
		resulting in a new signed graph $\Gamma'$, then $\lambda_1(\Gamma') \geq \lambda_1(\Gamma)$. 
		The equality holds only if the entries of $X$ corresponding to the 
		endpoints of these edges are all zeros.
	\end{lemma}
	
	\begin{lemma}[\cite{Stanic2019Ars}]\label{lemma:7}
		Let $\Gamma = (G, \sigma)$ be a signed graph. Then $\lambda_1(\Gamma) \leq \lambda_1(G)$. If $G$ is connected, then equality holds if and only if $\Gamma$ is balanced.
	\end{lemma}
	
	\begin{lemma}[\cite{Li2023}]\label{lemma:8}
		If $(K_n, \sigma)$ is unbalanced, then the shortest negative cycle must be $(C_3, \sigma|_{C_3})$.
	\end{lemma}
	
	A graph is Hamilton-connected if every pair of vertices is joined by a Hamiltonian path.
	
	\begin{lemma}[\cite{ZhouWang201}]\label{lemma:9}
		Let \(G\) be a connected graph on \(n \ge 6\) vertices with minimum degree \(\delta(G) \ge 3\).
		If
		\[
		\lambda_1(G) \ge \sqrt{(n-3)^2 + 10},
		\]
		then \(G\) is Hamilton-connected.
	\end{lemma}
	
	\begin{lemma}[Hong–Shu–Fang \cite{hong2001}]\label{lemma:10}
		Let $G$ be a simple connected graph with $n$ vertices, $m$ edges, minimum degree $\delta$, and spectral radius $\rho = \rho(G)$. Then
		\[
		\rho \le \frac{\delta - 1 + \sqrt{(\delta + 1)^2 + 4(2m - \delta n)}}{2}.
		\]
		
	\noindent Equality holds if and only if $G$ is regular or a bidegreed graph in which every vertex has degree either $\delta$ or $n-1$.
	\end{lemma}
	
	Set $C_{3,4}^{-}: = C_{3}^{-} \cup C_{4}^{-}$. For \(n \geq 5\), let \(\Gamma_n\) be the signed graph obtained from a copy of \(K_{n-2}\) by removing the edge \(v_3v_4\) and adding two new vertices \(v_1\) and \(v_2\) and three edges \(v_1v_2\), \(v_1v_3\) and \(v_2v_4\), where \(v_1v_2\) is the unique negative edge (see Fig.~\ref{fig:gamma_n}).
	
	\begin{lemma}[\cite{CaiZhou2024}]\label{lemma:11}
		Let $\Gamma$ be a $C_{3,4}^-$-free unbalanced signed graph of order $n \ge 6$. Then
		\[
		\rho(\Gamma) \le \gamma_n,
		\]
		with equality if and only if $\Gamma$ is switching isomorphic to $\Gamma_n$ (see Fig.~\ref{fig:gamma_n}), where $\gamma_n$ is the largest root of the equation
		\[
		\lambda^3 - (n-6)\lambda^2 - 3(n-4)\lambda - n + 3 = 0.
		\]
	\end{lemma}
	
	\begin{figure}[ht]
		\centering
		\begin{tikzpicture}[
			vertex/.style={circle, draw, fill=black, minimum size=8pt, inner sep=0pt},
			bigellipse/.style={ellipse, draw, thick, minimum width=4.8cm, minimum height=2.8cm},
			blackedge/.style={draw=black, thick}
			]
			% 椭圆 K_{n-2}，中心在 (0,0)
			\node[bigellipse] (K) at (0,-0.2) {};
			\node at (0,-0.8) {\textsf{\(K_{n-2} - v_3 v_4\)}};
			
			% v1, v2 位于椭圆上方
			\node[vertex, label={[label distance=2pt]above left:$v_1$}] (v1) at (-1.2, 1.8) {};
			\node[vertex, label={[label distance=2pt]above right:$v_2$}] (v2) at ( 1.2, 1.8) {};
			
			% 【修改这里】v3, v4 的标签改为 left 和 right，实现往下调
			\node[vertex, label={[label distance=2pt]left:$v_3$}] (v3) at (-1.2, 0.4) {};
			\node[vertex, label={[label distance=2pt]right:$v_4$}] (v4) at ( 1.2, 0.4) {};
			
			% 连接 v1-v3, v2-v4 黑边
			\draw[blackedge] (v1) -- (v3);
			\draw[blackedge] (v2) -- (v4);
			
			% 连接 v1-v2 红边，并加上负号
			\draw[red, thick] (v1) -- (v2) node[midway, below, text=black] {\scriptsize \(-\)};
		\end{tikzpicture}
		\caption{The signed graph \(\Gamma_n\) (red lines indicate negative edges).}
		\label{fig:gamma_n}
	\end{figure}
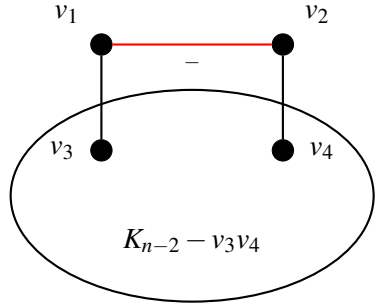
	
	\begin{lemma}[\cite{ZhouWang201}]\label{lem:ZhouWang}
		Let $G$ be a connected graph on $n \ge 6$ vertices with minimum degree $\delta(G) \ge 3$. If
		\[
		e(G) \ge \binom{n-2}{2} + 6,
		\]
		then $G$ is Hamilton-connected unless $G = K_3 \vee (K_{n-5} \cup I_2)$ or $G$ is one of the following exceptional graphs: when $n = 8$, $G \in \{ K_4 \vee I_4,\; K_3 \vee (K_{1,3} \cup K_1),\; K_3 \vee (K_{1,2} \cup K_2),\; K_2 \vee K_{2,4} \}$; when $n = 9$, $G = K_4 \vee (K_2 \cup I_3)$; when $n = 10$, $G \in \{ K_5 \vee I_5,\; K_4 \vee (K_{1,4} \cup K_1),\; K_3 \vee K_{2,5},\; K_4 \vee (K_{1,3} \cup K_2) \}$; when $n = 12$, $G = K_6 \vee I_6$.
	\end{lemma}

	\begin{proof}[\textbf{Proof of Theorem~\ref{thm:spectral}}]
		Suppose that $\Gamma = (G, \sigma)$ has the maximum index among all unbalanced 
		signed graphs of order $n$ containing no negative Hamiltonian cycle. By Lemma~\ref{lemma:4}, let $\Gamma' = (G, \sigma')$ be a signed graph switching equivalent to $\Gamma$ such that 
		$A(\Gamma')$ has a non-negative eigenvector corresponding to $\lambda_1(\Gamma')$. 
		So $\Gamma'$ is also unbalanced, contains no negative Hamiltonian cycle, and is extremal. Write $V(\Gamma') = \{v_1, v_2, \dots, v_n\}$, and let $X = (x_1, x_2, \dots, x_n)^{\top}$ be the non-negative unit eigenvector 
		of $A(\Gamma')$ corresponding to $\lambda_1(\Gamma')$, where $x_i$ corresponds to 
		the vertex $v_i$ for $1 \leq i \leq n$. Since $\Gamma_1$ is unbalanced and contains no negative Hamiltonian cycle, $\lambda_1(\Gamma') \geq \lambda_1(\Gamma_1) > n - \frac{5}{2}$ by 
		\hyperref[lemma:3ii]{(ii)} of Lemma~\ref{lemma:3}. 
		
		\begin{claim}\label{claim:1}
			$X$ has at most one zero coordinate.
		\end{claim}
		\begin{proof}
			Suppose that at least two coordinates of $X$ are zero. After relabeling the vertices, we may assume that $x_1 = x_2 = 0$. Then
			\begin{align*}
				\lambda_1(\Gamma') &= X^{\top} A(\Gamma')X =(x_3,x_4, \dots,x_n)A(\Gamma'- v_1 - v_2)(x_3,x_4, \dots,x_n)^{\top}\\
				&\leq \lambda_1(\Gamma' - v_1 - v_2)\\
				& \leq \lambda_1(K_{n-2}) \\
				&= n-3,
			\end{align*}
			a contradiction. 
		\end{proof}
		
		\begin{claim}\label{claim:2}
			$\Gamma'$ is connected.
		\end{claim}
		\begin{proof}
			Otherwise, assume that $\Gamma'_1$ and $\Gamma'_2$ are two distinct connected 
			components of $\Gamma'$ with $\lambda_1(\Gamma') = \lambda_1(\Gamma'_1)$. 
			Let $v_i \in V(\Gamma'_1)$ and $v_j \in V(\Gamma'_2)$. Then we can construct a new signed graph $\Gamma^*$ from $\Gamma'$ by adding a positive edge $v_i v_j$. Clearly, $\Gamma^*$ is also unbalanced signed graph contains no negative Hamiltonian cycle and 
			\begin{align*}
				\lambda_1(A(\Gamma^*))-\lambda_1(A(\Gamma'))
				&\ge X^{T}\bigl(A(\Gamma^*)-A(\Gamma')\bigr)X \\
				&= 2x_i x_j \ge 0.
			\end{align*}
			 If \(\lambda_1(A(\Gamma^*))=\lambda_1(A(\Gamma'))\), then \(X\) is also an eigenvector of \(A(\Gamma^*)\) corresponding to \(\lambda_1(A(\Gamma^*))\). By \hyperref[lemma:6i]{(i)} of Lemma~\ref{lemma:6}, we obtain that \(x_i=x_j=0\), which contradicts Claim ~\ref{claim:1}. Thus, $\lambda_1(\Gamma^*) > \lambda_1(\Gamma')$, a contradiction.
		\end{proof}
		
		Let $\mathcal{C}$ be the shortest negative cycle of $\Gamma'$ and $ l = |V(C)|$.
		
		\begin{claim}\label{claim:3}
			$\mathcal{C}$ contains all negative edges of $\Gamma'$.
		\end{claim}
		\begin{proof}
			Otherwise, assume that $e = v_i v_j$ is a negative edge of $\Gamma'$ and $e \notin E(\mathcal{C})$. Then we can construct a new signed graph $\Gamma^{*}$ from $\Gamma'$ by deleting $e$. Clearly, $\Gamma^{*}$ is also unbalanced signed graph contains no negative Hamiltonian cycle and 
			\begin{align*}
				\lambda_1(A(\Gamma^*))-\lambda_1(A(\Gamma'))
				&\ge X^{T}\bigl(A(\Gamma^*)-A(\Gamma')\bigr)X \\
				&= 2x_i x_j \ge 0.
			\end{align*}
			If \(\lambda_1(A(\Gamma^*))=\lambda_1(A(\Gamma'))\), then \(X\) is also an eigenvector of \(A(\Gamma^*)\) corresponding to \(\lambda_1(A(\Gamma^*))\). By \hyperref[lemma:6ii]{(ii)} of Lemma~\ref{lemma:6}, we obtain that \(x_i=x_j=0\), which contradicts Claim ~\ref{claim:1}. Hence, $\lambda_1(\Gamma^*) > \lambda_1(\Gamma')$, a contradiction.
			 
		\end{proof}
		
		\begin{claim}\label{claim:4}
			$\delta(\Gamma') \ge 2$.
		\end{claim}
		\begin{proof}
			Otherwise, assume that there exists a vertex $u$ such that $d(u)=1$ and $N(u)=\{v\}$. We  first claim that $\Gamma' - u$ is a signed complete graph. Otherwise, assume that there exist vertices $v_i, v_j \in V(\Gamma'-u)$ such that $v_i v_j \notin E(\Gamma'-u)$. Then we can construct a new signed graph $\Gamma^*$ from $\Gamma'$ by adding a positive edge $v_i v_j$. Clearly, $\Gamma^*$ is also unbalanced signed graph contains no negative Hamiltonian cycle and 
			\begin{align*}
				\lambda_1(A(\Gamma^*))-\lambda_1(A(\Gamma'))
				&\ge X^{T}\bigl(A(\Gamma^*)-A(\Gamma')\bigr)X \\
				&= 2x_i x_j \ge 0.
			\end{align*}
			If \(\lambda_1(A(\Gamma^*))=\lambda_1(A(\Gamma'))\), then \(X\) is also an eigenvector of \(A(\Gamma^*)\) corresponding to \(\lambda_1(A(\Gamma^*))\). By \hyperref[lemma:6i]{(i)} of Lemma~\ref{lemma:6}, we obtain that \(x_i=x_j=0\), which contradicts Claim ~\ref{claim:1}. Therefore, $\lambda_1(\Gamma^*) > \lambda_1(\Gamma')$, a contradiction.
			
			Next, we assert that $\Gamma' - u$ has exactly one negative edge. Claim~\ref{claim:3} shows that all negative edges lie on a shortest negative cycle of $\Gamma'-u$. Since $\Gamma' - u$ is an unbalanced signed complete graph, it must have one or three negative edges
			by Lemma~\ref{lemma:8}. If $\Gamma'-u$ has  three negative edges: $v_iv_j,v_jv_k,v_kv_i$, then we can construct a new signed graph $\Gamma^*$ from $\Gamma'$ by reversing the signs of $v_iv_j,v_jv_k$ on $\mathcal{C}$. Clearly, $\Gamma^*$ is also unbalanced signed graph contains no negative Hamiltonian cycle and 
			\begin{align*}
				\lambda_1(A(\Gamma^*))-\lambda_1(A(\Gamma'))
				&\ge X^{T}\bigl(A(\Gamma^*)-A(\Gamma')\bigr)X \\
				&= 4(x_i x_j+x_j x_k) \ge 0.
			\end{align*}
			If \(\lambda_1(A(\Gamma^*))=\lambda_1(A(\Gamma'))\), then \(X\) is also an eigenvector of \(A(\Gamma^*)\) corresponding to \(\lambda_1(A(\Gamma^*))\). By \hyperref[lemma:6iii]{(iii)} of Lemma~\ref{lemma:6}, we obtain that \(x_i=x_j=x_k=0\), which contradicts Claim ~\ref{claim:1}. Thus, $\lambda_1(\Gamma^*) > \lambda_1(\Gamma')$, a contradiction.
			Therefore, $\Gamma' - u$ is a signed complete graph with exactly one negative edge, denoted by $e_2 = wz$. Set $e_1 = uv$.
			
			If $e_1 \cap e_2 = \{v\}$, then we can construct a new signed graph $\Gamma^*$ from $\Gamma'$ by adding a positive edge $uz$ or $uw$. Without loss of generality, we adding the positive edge $uz$.
			Clearly, $\Gamma^*$ is also unbalanced signed graph contains no negative Hamiltonian cycle and 
			\begin{align*}
				\lambda_1(A(\Gamma^*))-\lambda_1(A(\Gamma'))
				&\ge X^{T}\bigl(A(\Gamma^*)-A(\Gamma')\bigr)X \\
				&= 2x_u x_z \ge 0.
			\end{align*}
			If \(\lambda_1(A(\Gamma^*))=\lambda_1(A(\Gamma'))\), then \(X\) is also an eigenvector of \(A(\Gamma^*)\) corresponding to \(\lambda_1(A(\Gamma^*))\). By \hyperref[lemma:6i]{(i)} of Lemma~\ref{lemma:6}, we obtain that \(x_u=x_z=0\), which contradicts Claim ~\ref{claim:1}. Hence, $\lambda_1(\Gamma^*) > \lambda_1(\Gamma')$, a contradiction.
			
			If $e_1 \cap e_2 = \emptyset$, then
			\[
			A(\Gamma') = \begin{bmatrix}
				0 & 1 & \bm{0} & \bm{0} \\
				1 & 0 & J & J \\
				\bm{0} & J & I-J & J \\
				\bm{0} & J & J & J - I
			\end{bmatrix}
			\begin{array}{l}
				u \\
				v \\
				\{w,z\} \\
				V(\Gamma')\setminus\{u,v,w,z\}
			\end{array}
			\]
			and $A(\Gamma')$ has the equitable quotient matrix
			\[
			Q_4 = \begin{bmatrix}
				0 & 1 & 0 & 0 \\
				1 & 0 & 2 & n-4 \\
				0 & 1 & -1 & n-4 \\
				0 & 1 & 2 & n-5
			\end{bmatrix}.
			\]
			By calculation, the characteristic polynomial of $Q_4$ is
			\[
			P_{Q_4}(x) = x^4+(6-n)x^3+(14-4n)x^2+(4-2n)x+3n-13.
			\]
			
			\noindent By Lemma~\ref{lemma:1}, all eigenvalues of $Q_4$ belong to the spectrum of $A(\Gamma')$ and by subtracting a suitable scalar multiple of the all-ones block from each subblock of $A(\Gamma')$, we derive the simplified block matrix
			\[
			A_4 =
			\begin{bmatrix}
				0 & 0 & 0 & \bm{0} \\    
				0 & 0 & 0 & \bm{0} \\    
				0 & 0 & I_2 & \bm{0} \\  
				\bm{0} & \bm{0} & \bm{0} & -I_{n-4}
			\end{bmatrix}.
			\]
			
		\noindent  Since $\lambda_1(\Gamma_1) > 1 $, $\operatorname{Spec}(A_4) = \{-1^{[n-4]}, 0^{[2]}, 1^{[2]}\}$ and $\lambda_1(\Gamma') = \lambda_1(Q_4)$. Recall that \(P_{Q_1}(x)\) is the characteristic polynomial of the equitable quotient matrix $Q_1$ of \(\Gamma_1\). A direct computation gives
			
			\[
			P_2(x)=P_{Q_1}(x)-	P_{Q_4}(x) = -2x^3 + (2n-11)x^2 + (7n-20)x + 21 - 5n.
			\]
			Let $\lambda_1 = \lambda_1(\Gamma_1) > n-\frac{5}{2} \geq \frac{11}{2}$ be the largest real root of \(	P_{Q_1}(x)\) from \hyperref[lemma:3i]{(i)} of Lemma~\ref{lemma:3}. Since \(\lambda_1\) is the largest real root of \(P_{Q_1}(x)=(x-1)\bigl(x^3+(5-n)x^2+(8-3n)x+2n-8\bigr)\) and \(P_{Q_1}\) has positive leading coefficient, we have \(P_{Q_1}(x)\ge 0\) for every \(x\ge \lambda_1\), with equality only at \(x=\lambda_1\).  Write \(P_{Q_1}(x)=A(x)+nB(x)\), where \(A(x)=x^4+4x^3+3x^2-16x+8\) and \(B(x)=-x^3-2x^2+5x-2\). \(B(\lambda_1) < 0\) since \(\lambda_1>\frac{11}{2}\). Solving \(	P_{Q_1}(\lambda_1)=0\) for \(n\) yields
			\[
			n = -\frac{A(\lambda_1)}{B(\lambda_1)} = \frac{\lambda_1^3 + 5\lambda_1^2 + 8\lambda_1 - 8}{\lambda_1^2 + 3\lambda_1 - 2}.
			\]
			Note that
			\[
			P_{Q_4}(\lambda_1) = 2\lambda_1^3 + (11-2n)\lambda_1^2 + (20-7n)\lambda_1 + 5n - 21.
			\]
			 Substituting this expression for \(n\) and
			simplifying, we obtain
			\[
			P_{Q_4}(\lambda_1) = \frac{3\lambda_1^3 + 2\lambda_1^2 - 7\lambda_1 + 2}{\lambda_1^2 + 3\lambda_1 - 2}.
			\]
			Set
			\[
			M(\lambda_1)=3\lambda_1^3+2\lambda_1^2-7\lambda_1+2.
			\]
			Since \(M'(\lambda_1)=9\lambda_1^2+4\lambda_1-7>0\) for \(\lambda_1>1\), \(M\) is increasing for \(\lambda_1>1\). 
			Since \(M(\frac{11}{2})>0\), we have \(M(\lambda_1)>0\). Hence \(P_{Q_4}(\lambda_1) > 0\). Furthermore, \(P_2''(x) = -12x + 4n - 22\). When \(x \geq n - 3\), \(P_2''(x) \leq -8n + 14 < 0\).  \(P_2'(n - 3) = -2n^2 + 9n - 8 < 0\) (for \(n > 8\)). So \(P_2'\) is always negative on \([n - 3, \infty)\), and \(P_2\) is strictly decreasing.  Since $\lambda_1  > n-\frac{5}{2} > n-3$, we obtain that $P_2(x) < 0$ for $x \geq \lambda_1$. Recall that \(P_{Q_1}(x)\ge 0\) for \(x\ge \lambda_1\), it follows that \(P_{Q_4}(x)=P_{Q_1}(x)-P_2(x)>0\) for \(x\ge \lambda_1\). Thus \(P_{Q_4}\) has no real zero in \([\lambda_1,\infty)\). As all other eigenvalues of \(Q_4\) are strictly less than \(\lambda_1\), we obtain \(\lambda_1(\Gamma')=\lambda_1(Q_4)<\lambda_1=\lambda_1(\Gamma_1)\), a contradiction.
		 
		\end{proof}
	\end{proof}
	
	\vspace{-3em}
	
	\begin{claim}\label{claim:5}
		The underlying graph $G$ contains a Hamiltonian cycle.
	\end{claim}
	\begin{proof}
		By Claim \ref{claim:4} and $\lambda_1(G) \ge \lambda_1(\Gamma') \geq \lambda_1(\Gamma_1) > n - \frac{5}{2} > n - 3$, we know that \(G\) has a Hamiltonian cycle unless \(G\) is one of the following four exceptional graphs: $H_1 = K_2 \vee (K_{n-4} \cup I_2)$ or 
		$H_2 = K_1 \vee (K_{n-3} \cup K_2)$, $H_3 =I_5 \vee K_4$ or $H_4=K_3 \vee (K_{1,4} \cup K_1) $. By Lemma~\ref{lemma:3} \hyperref[lemma:3ii]{(ii)} and \hyperref[lemma:3iii]{(iii)}, \hfill $\lambda_1(H_1)$, $\lambda_1(H_2) < n - \frac{5}{2}$. A direct computation gives \(\lambda_1(H_3) = \frac{3+\sqrt{89}}{2} \approx 6.217<\frac{13}{2}\) and \(\lambda_1(H_4) \approx 6.032<\frac{13}{2}\). Therefore \(G\) is not one of the above four exceptional graphs. Thus the underlying graph \(G\) contains a Hamiltonian cycle.
	\end{proof}
	
	Let $e^-(\Gamma')$ denote the number of negative edges of $\Gamma'$.
	
	\begin{claim}
		$e^-(\Gamma')=1$ or $3$.\label{claim:11}
	\end{claim}
	\begin{proof}
		Recall that $\mathcal{C}$ is the shortest negative cycle of $\Gamma'$. We first show that $l \le 4$. Suppose, for a contradiction, that $l \ge 5$. Then $\Gamma'$ is both $C_3^{-}$-free  and  $C_4^{-}$-free. Hence $\lambda_1(\Gamma') \le \rho(\Gamma') \le \gamma_n$, by Lemma~\ref{lemma:11}. Consider 
		\[
		f(x) = x^3 - (n-6)x^2 - 3(n-4)x - n + 3, \quad n > 8.
		\]
		Its derivative is
		\[
		f'(x) = 3x^2 - 2(n-6)x - 3(n-4).
		\]
		The critical points are the roots of \(f'(x)=0\):
		\[
		x = \frac{2(n-6) \pm \sqrt{4(n-6)^2 + 36(n-4)}}{6}
		= \frac{(n-6) \pm \sqrt{n(n-3)}}{3}.
		\]
		Let \(x_2 = \frac{(n-6) + \sqrt{n(n-3)}}{3}\) be the larger critical point.
		Since the leading coefficient of \(f'\) is positive, \(f\) is increasing for \(x \le x_1\) and for \(x \ge x_2\), and decreasing for \(x_1 \le x \le x_2\), where \(x_1\) is the smaller root. We claim that \(x_2 < n-3\) for all \(n > 8\). Indeed,
		\[
		x_2 - (n-3) = \frac{\sqrt{n(n-3)} - (2n-3)}{3}.
		\]
		For \(n > 8\), we have \(2n-3 > 0\), and
		\[
		(2n-3)^2 - n(n-3) = 3n^2 - 9n + 9 = 3(n^2 - 3n + 3) > 0,
		\]
		hence \(2n-3 > \sqrt{n(n-3)}\), which implies \(x_2 < n-3\). Therefore, for all \(x \ge n-3\), we have \(x \ge x_2\), so \(f\) is strictly increasing on the domain \(x \ge n-3\). 
		It is easy to verify that for \(n > 8\),
		\[
		f(n-3) = 2n - 6 > 0.
		\]
		Since \(f(\gamma_n)=0\) and \(f\) is strictly increasing for all \(x \ge n-3\), we conclude that \(\gamma_n < n-3\) for \(n > 8\). Consequently $\lambda_1(\Gamma') \le \gamma_n < n-3$, a contradiction. Thus $l \le 4$. Since all negative edges lie on the shortest negative cycle \(\mathcal{C}\), and  \(l \in \{3, 4\}\), the number of negative edges is odd and at most three. Therefore, \(e^{-}(\Gamma') \in \{1, 3\}\).

	\end{proof}

	\begin{claim}\label{claim:6}
		All edges of every Hamiltonian cycle are positive.
	\end{claim}
	\begin{proof}
		The assertion is immediate when \(e^{-}(\Gamma') = 1\). Suppose that \(e^{-}(\Gamma') = 3\), and let \(H\) be a Hamiltonian cycle of \(G\). Since \(\Gamma'\) contains no negative Hamiltonian cycle, \(H\) contains either zero or two negative edges.
		
		First suppose that the shortest negative cycle is a triangle \(xyzx\). Then all three edges \(xy, yz, zx\) are negative. If \(H\) contains negative edges, it contains exactly two of them; after relabeling, assume that \(xz, yz \in E(H)\), whereas \(xy \notin E(H)\). The edge \(xy\) is therefore a chord of \(H\). The two edges \(xz\) and \(zy\) form one \(x\)-\(y\) path of \(H\), and the complementary \(x\)-\(y\) path \(P\) is entirely positive. Hence $xy \cup P$ is a negative cycle. Then, we can construct a new signed graph $\Gamma^*$ from $\Gamma'$ by deleting the negative edge $xz$ or $yz$. Without loss of generality, we deleting the negative  edge $yz$. Clearly, $\Gamma^*$ is also unbalanced signed graph contains no negative Hamiltonian cycle and 
		\begin{align*}
			\lambda_1(A(\Gamma^*))-\lambda_1(A(\Gamma'))
			&\ge X^{T}\bigl(A(\Gamma^*)-A(\Gamma')\bigr)X \\
			&= 2x_y x_z \ge 0.
		\end{align*}
		If \(\lambda_1(A(\Gamma^*))=\lambda_1(A(\Gamma'))\), then \(X\) is also an eigenvector of \(A(\Gamma^*)\) corresponding to \(\lambda_1(A(\Gamma^*))\). By \hyperref[lemma:6ii]{(ii)} of Lemma~\ref{lemma:6}, we obtain that \(x_y=x_z=0\), which contradicts Claim ~\ref{claim:1}. Hence, $\lambda_1(\Gamma^*) > \lambda_1(\Gamma')$, a contradiction. Now suppose that the shortest negative cycle is $C = xyzwx$, where \(xy, yz, zw\) are negative and \(wx\) is positive. If \(H\) contains negative edges, it contains exactly two of \(xy, yz, zw\); the third one is a chord of \(H\).
		
		If the chord is \(xy\) or \(zw\), the other two negative edges lie on the same path determined by the chord. The chord together with the other, entirely positive path forms a negative cycle. Deleting one of the two negative edges on the first path preserves this negative cycle and gives the required contradiction.
		
		Suppose next that the chord is \(yz\). If \(xy\) and \(zw\) lie on the same \(y\)-\(z\) path of \(H\), the preceding argument applies. Otherwise, let \(P_1\) be the \(y\)-\(z\) path containing \(xy\), and let \(P_2\) be the \(y\)-\(z\) path containing \(zw\). Then \(P_2 - zw\) is a positive \(y\)-\(w\) path. Consequently, $yx \cup xw \cup (P_2 - zw)$ is a negative cycle. Then, we can construct a new signed graph $\Gamma^*$ from $\Gamma'$ by deleting the negative edge $yz$.  Clearly, $\Gamma^*$ is also unbalanced signed graph contains no negative Hamiltonian cycle and 
		\begin{align*}
			\lambda_1(A(\Gamma^*))-\lambda_1(A(\Gamma'))
			&\ge X^{T}\bigl(A(\Gamma^*)-A(\Gamma')\bigr)X \\
			&= 2x_i x_j \ge 0.
		\end{align*}
		If \(\lambda_1(A(\Gamma^*))=\lambda_1(A(\Gamma'))\), then \(X\) is also an eigenvector of \(A(\Gamma^*)\) corresponding to \(\lambda_1(A(\Gamma^*))\). By \hyperref[lemma:6ii]{(ii)} of Lemma~\ref{lemma:6}, we obtain that \(x_i=x_j=0\), which contradicts Claim ~\ref{claim:1}. Thus, $\lambda_1(\Gamma^*) > \lambda_1(\Gamma')$, a contradiction.
	\end{proof}
	
	\begin{claim}\label{claim:7}
		$\Gamma'$ contains exactly one negative edge.
	\end{claim}
	\begin{proof}
	    Suppose, for a contradiction, that \(\Gamma'\) has more than one negative edge. By Claim~\ref{claim:11}, it has exactly three. Let $e_2=v_2v_3$ and $e_3=v_1v_3$ be two negative edges of $\Gamma'$. Then we can construct a new signed graph $\Gamma^*$ from $\Gamma'$ by reversing the signs of $e_2$ and $e_3$. After this operation, the original shortest negative cycle contains exactly one negative edge. Thus $\Gamma^*$ is also unbalanced signed graph contains no negative Hamiltonian cycle and 
	    \begin{align*}
	    	\lambda_1(A(\Gamma^*))-\lambda_1(A(\Gamma'))
	    	&\ge X^{T}\bigl(A(\Gamma^*)-A(\Gamma')\bigr)X \\
	    	&= 4(x_1 +x_2 )x_3 \ge 0.
	    \end{align*}
	    If \(\lambda_1(A(\Gamma^*))=\lambda_1(A(\Gamma'))\), then \(X\) is also an eigenvector of \(A(\Gamma^*)\) corresponding to \(\lambda_1(A(\Gamma^*))\). By \hyperref[lemma:6iii]{(iii)} of Lemma~\ref{lemma:6}, we obtain that \(x_1=x_2=x_3=0\), which contradicts Claim ~\ref{claim:1}. Thus, $\lambda_1(\Gamma^*) > \lambda_1(\Gamma')$, a contradiction.  
	\end{proof}
	
	Let $v_1v_2$ be the unique negative edge of $\Gamma'$.
	
	\begin{claim}\label{claim:8}
		$\delta(\Gamma') = 2$.
	\end{claim}
	\begin{proof}
		
		Suppose for contradiction that \(\delta(\Gamma') \ge 3\). Define 
		\[
		f(x) = x^3 + (5 - n)x^2 + (8 - 3n)x + 2n - 8.
		\]
		Note that \(\lambda_1(\Gamma_1)\) is the largest real root of \(f(x) = 0\). For every \(n > 8\), direct verification shows 
		\[
		f\left(n - \frac{19}{8}\right) = \frac{-192n^2 + 2384n - 6243}{512}< 0 \quad \text{and} \quad f(n-2) =4(n-3)> 0.
		\]
		So \(f(x) = 0\) has a real root \(\lambda^* \in \left(n - \frac{19}{8}, n-2\right)\). Consequently,
		\[
		\lambda_1(\Gamma') \ge \lambda_1(\Gamma_1) \ge \lambda^* > n - \frac{19}{8}\quad\text{for all }n > 8.
		\]  
		
		\noindent If \(n \ge 11\), then 
		\[
		\lambda_1(G) \ge \lambda_1(\Gamma') > n - \frac{19}{8} > \sqrt{(n-3)^2 + 10}.
		\]
		By Lemma~\ref{lemma:9}, $G$ is Hamilton-connected. So there exists a Hamiltonian path between \(v_1\) and \(v_2\), and consequently \(\Gamma'\) contains a negative Hamiltonian cycle, a contradiction. Hence \(\delta(\Gamma') = 2\). Now suppose that $8 < n \le 10$, and \(\delta(G) \ge 3\). Let $m = e(G)$, $\rho = \rho(G)$.
			
			\noindent By Lemma~\ref{lemma:10}, we have
			\[
			2m \ge \rho^2 + \rho + \delta(n - 1 - \rho).
			\]
			Since \(\rho \le n - 1\) and \(\delta \ge 3\),
			\[
			2m \ge \rho^2 + \rho + 3(n - 1 - \rho) = (\rho - 1)^2 + 3n - 4.
			\]
			Since
			\[
			\rho \ge \lambda_1(\Gamma') > n - \frac{5}{2},
			\]
			we obtain
			\[
			2m > \left(n - \frac{7}{2}\right)^2 + 3n - 4.
			\]
			Moreover
			\[
			\left(n - \frac{7}{2}\right)^2 + 3n - 4 - 2\left(\binom{n-2}{2} + 5\right) = \frac{4n - 31}{4} > 0.
			\]
			Therefore
			\[
			m \ge \binom{n-2}{2} + 6.
			\]
			By Lemma~\ref{lem:ZhouWang}, \(G\) is Hamilton-connected unless \(G\) belongs to its list of exceptions.

			\begin{table}[h]
				\centering
				\caption{Polynomial factors and numerical values for the listed graphs.}
				\label{tab:polynomial-values}
				\begin{tabular}{c l l l}
					\toprule
					\(n\) & Graph & \(p_H(x)\) & \(\lambda_1\) (4 d.p.) \\
					\midrule
					9 & \(\Gamma_1\) & \(x^3 - 4x^2 - 19x + 10\) & 6.6361 \\
					9 & \(K_4 \vee (K_2 \cup I_3)\) & \(x^2 - 7x + 4\) & 6.3723 \\
					9 & \(K_3 \vee (K_4 \cup I_2)\) & \(x^3 - 5x^2 - 12x + 18\) & 6.4308 \\
					10 & \(\Gamma_1\) & \(x^3 - 5x^2 - 22x + 12\) & 7.6657 \\
					10 & \(K_5 \vee I_5\) & \(x^2 - 4x - 25\) & 7.3852 \\
					10 & \(K_4 \vee (K_{1,4} \cup K_1)\) & \(x^4 - 3x^3 - 28x^2 - 20x + 16\) & 7.2196 \\
					10 & \(K_3 \vee K_{2,5}\) & \(x^3 - 2x^2 - 31x - 40\) & 7.1325 \\
					10 & \(K_4 \vee (K_{1,3} \cup K_2)\) & \(x^4 - 4x^3 - 24x^2 + 4x + 39\) & 7.1655 \\
					10 & \(K_3 \vee (K_5 \cup I_2)\) & \(x^3 - 6x^2 - 13x + 24\) & 7.3272 \\
					\bottomrule
				\end{tabular}
			\end{table}
			
			Let \(p_H(x)\) be the characteristic polynomial of the equitable quotient matrix of these exceptional graphs and \(\Gamma_1\), and let \(\lambda_1\) denote the largest eigenvalue. Note that from the data computed in the table above, we can conclude that none of the exceptions can satisfy $\lambda_1(G) \ge \lambda_1(\Gamma_1)$. Therefore, the underlying graph of the extremal graph \(\Gamma'\) cannot be any of these exceptional graphs. Hence \(G\) is Hamilton-connected, which again yields a negative Hamiltonian cycle, a contradiction. Finally, for all \(n > 8\), $\delta(G) = 2$.

	\end{proof}
	\vspace{-1em}
	
	A chord of a cycle C is an edge between two non-consecutive vertices of the cycle. Note that the unique negative edge is a chord of every Hamiltonian cycle of $G$. Let $w$ be a vertex of degree two. 
	
	\begin{claim}\label{claim:9}
		\(N_G(w) = \{v_1, v_2\}\).
	\end{claim}

		\begin{proof}
			
		Since \(v_1v_2\) is not contained in any Hamiltonian cycle, fix a Hamiltonian cycle and write it as
		\[
		H = v_1 a_1 a_2 \cdots a_p v_2 b_q b_{q-1} \cdots b_1 v_1, \qquad p, q \ge 1.
		\]
		If \(a_1 b_q \in E(G)\), then
		\[
		v_1 b_1 b_2 \cdots b_q a_1 a_2 \cdots a_p v_2
		\]
		is a Hamiltonian \(v_1\)-\(v_2\) path. Together with the negative edge \(v_1v_2\), it forms a negative Hamiltonian cycle, a contradiction. Hence
		\begin{equation} \tag{2.1}\label{2.1}
			a_1 b_q \notin E(G).
		\end{equation}
		Similarly,
		\begin{equation} \tag{2.2}\label{2.2}
			b_1 a_p \notin E(G).
		\end{equation}
		Since \(v_1v_2\) is a chord of \(H\), \(d_G(v_1), d_G(v_2) \ge 3\). Therefore \(w \notin \{v_1, v_2\}\). Its two neighbors are precisely its two neighbors on \(H\).
		
		Suppose, contrary to the claim, that \(N_G(w) \neq \{v_1, v_2\}\). First, we assert that at least one of the two nonedges in \eqref{2.1} and \eqref{2.2} is not incident with \(w\). Indeed, if both were incident with \(w\), then either \(p=1\) and \(w = a_1 = a_p\), or \(q=1\) and \(w = b_1 = b_q\). In either case, the two neighbors of \(w\) on \(H\) would be \(v_1\) and \(v_2\), contrary to our assumption. Fix such a nonedge and denote it by \(rs\).
		
		We construct a new signed graph \(S\) from \(\Gamma'\) as follows. All original edges and their signs are retained. Among the vertices other than \(w\), every missing edge except \(rs\) is added as a positive edge. The edges incident with \(w\) are specified below. Recall that we assume \(N_G(w) \neq \{v_1, v_2\}\), so \(w\) is adjacent to neither $v_1$ nor $v_2$, or to exactly one of them. If \(w\) is adjacent to neither \(v_1\) nor \(v_2\), write \(N_{\Gamma'}(w) = \{a, b\}\) and add the positive edge \(wv_1\). If \(w\) is adjacent to exactly one of them, say \(v_1\), write \(N_{\Gamma'}(w) = \{v_1, a\}\), and we choose an arbitrary vertex \(b \in V(\Gamma') \setminus \{v_1, v_2, w, a, r, s\}\) (which is possible since \(n > 8\)) and add the positive edge \(wb\). In either case, \(N_S(w) = \{v_1, a, b\}\). Furthermore, among the vertices different from \(w\), the only missing edge is \(rs\). If \(w\) is adjacent to neither \(v_1\) nor \(v_2\), then set of newly added positive edges is
		\[
		E^+_{\mathrm{new}}
		=
		\{uv\notin E(\Gamma'): u,v\neq w,\ uv\neq rs\}
		\cup
		\{wv_1\}.
		\] 
		If \(w\) is adjacent to exactly one of \(v_1, v_2\), then set of newly added positive edges is
		\[
		E^+_{\mathrm{new}}
		=
		\{uv\notin E(\Gamma'): u,v\neq w,\ uv\neq rs\}
		\cup
		\{wb\}.
		\] 
	 In either case, the Rayleigh quotient difference can be written in the same form
		\[
		\lambda_1(S)-\lambda_1(\Gamma')
		\ge
		X^\top\bigl(A(S)-A(\Gamma')\bigr)X
		=
		2\sum_{uv\in E^+_{\mathrm{new}}}x_ux_v
		\ge 0.
		\]
		If \(\lambda_1(A(S))=\lambda_1(A(\Gamma'))\), then \(X\) is also an eigenvector of \(A(S)\) corresponding to \(\lambda_1(A(S))\). By \hyperref[lemma:6i]{(i)} of Lemma~\ref{lemma:6}, we at least get \(x_u=x_v=0\), which contradicts Claim ~\ref{claim:1}. Therefore, \(\lambda_1(\Gamma')<\lambda_1(S)\), a contradiction.
		
		If \(w\) is adjacent to neither \(v_1\) nor \(v_2\), then \(w\) lies in the interior of one of the two \(v_1\)-\(v_2\) arcs of \(H\), and both of its neighbors \(a, b\) lie on that same arc. Each of the two forced nonedges has one endpoint on each \(v_1\)-\(v_2\) arc. Hence neither forced nonedge can have both endpoints in \(\{a, b\}\). If \(w\) is adjacent to exactly one of \(v_1, v_2\), the newly chosen vertex \(b\) is outside \(\{r, s\}\), so again the intersection has size at most one. Therefore, by the cyclic ordering of \(H\),
		\[
		|\{r, s\} \cap \{a, b\}| \le 1.
		\]
		
			\noindent \textbf{Case 1.} \(\{r, s\} \cap \{a, b\} = \varnothing\).
			
			With respect to the partition
			\[
			\{w\}, \quad \{a, b\}, \quad \{v_1\}, \quad \{v_2\}, \quad \{r, s\}, \quad R, \qquad |R| = n-7,
			\]
			the quotient matrix is
			\[
			Q_5 = \begin{bmatrix}
				0 & 2 & 1 & 0 & 0 & 0 \\
				1 & 1 & 1 & 1 & 2 & n-7 \\
				1 & 2 & 0 & -1 & 2 & n-7 \\
				0 & 2 & -1 & 0 & 2 & n-7 \\
				0 & 2 & 1 & 1 & 0 & n-7 \\
				0 & 2 & 1 & 1 & 2 & n-8
			\end{bmatrix}.
			\]
			Its characteristic polynomial is
			\begin{align*}
				P_{Q_5}(x) &= x^6 + (7-n)x^5 + (18-5n)x^4 - (2n+8)x^3 \\
				&\quad + (17n-95)x^2 + (13n-55)x + 76 - 10n.
			\end{align*}
			
		\noindent	By construction, every vertex in a displayed cell has the same signed row sum toward each other cell. Hence the displayed partition is equitable, and Lemma~\ref{lemma:1} applies. By Lemma~\ref{lemma:1}, all eigenvalues of \(Q_5\) belong to the spectrum of \(A(S)\), and by subtracting a suitable scalar multiple of the all-ones block from each subblock of \(A(S)\), we derive the simplified block matrix
			\[
			A_5=
			\begin{bmatrix}
				0 & \bm{0} & 0 & 0 & \bm{0} & \bm{0}\\
				\bm{0} & -I_2 & \bm{0} & \bm{0} & \bm{0} & \bm{0}\\
				0 & \bm{0} & 0 & 0 & \bm{0} & \bm{0}\\
				0 & \bm{0} & 0 & 0 & \bm{0} & \bm{0}\\
				\bm{0} & \bm{0} & \bm{0} & \bm{0} & \bm{0} & \bm{0}\\
				\bm{0} & \bm{0} & \bm{0} & \bm{0} & \bm{0} & -I_{n-7}
			\end{bmatrix}.
			\]
			Since \(\lambda_1(S)>0\), \(\operatorname{Spec}(A_5)=\{-1^{[n-5]},0^{[5]}\}\) and \(\lambda_1(S)=\lambda_1(Q_5)\).
			
			\noindent Recall that
			\[
			P_{Q_1}(x)=(x-1)\bigl(x^3+(5-n)x^2+(8-3n)x+2n-8\bigr).
			\]
			Define
			\[
			\begin{aligned}
				P_3(x)&=x^2P_{Q_1}(x)-P_{Q_5}(x)\\
				&= -3x^5 + (3n-15)x^4 + (7n-8)x^3\\
				&\qquad + (103-19n)x^2 + (55-13n)x + 10n - 76.
			\end{aligned}
			\]
			Substituting \(x=n-\frac{5}{2}\) into \(P_3(x)\) and writing \(\eta=n-9\ge 0\), we obtain
			\[
			P_3\!\left(n-\frac{5}{2}\right)
			=-\frac{\eta^4}{2}-\frac{45\eta^3}{2}-\frac{539\eta^2}{2}
			-\frac{9347\eta}{8}-\frac{49459}{32}<0.
			\]
			
			\noindent The derivative of \(P_3(x)\) is
			\[
			\begin{aligned}
				P_3'(x)&=-15x^4+(12n-60)x^3+(21n-24)x^2+(206-38n)x+(55-13n)\\
				&=\bigl(12x^3+21x^2-38x-13\bigr)n
				+\bigl(-15x^4-60x^3-24x^2+206x+55\bigr).
			\end{aligned}
			\]
			For \(x\ge 5\), the coefficient of \(n\) satisfies \(12x^3+21x^2-38x-13>0\). Hence \(P_3'(x)\) is strictly increasing in \(n\). Since \(x\ge n-3\), i.e., \(n\le x+3\),
			\[
			\begin{aligned}
				P_3'(x)&\le P_3'(x)\big|_{n=x+3}\\
				&=-3x^4-3x^3+x^2+79x+16.
			\end{aligned}
			\]
			Define \(H(x)=-3x^4-3x^3+x^2+79x+16\). Writing \(t=x-5\ge 0\), we have
			\[
			H(5+t)=-3t^4-63t^3-494t^2-1636t-1814<0,
			\]
			so \(H(x)<0\) for all \(x\ge 5\). Consequently, \(H(x)\le H(5)<0\) for all \(x\ge n-3\ge 5\). This means that
			\[
			P_3'(x)\le H(x)<0.
			\]
			Thus \(P_3(x)\) is strictly decreasing for all \(x\ge n-3\). Recall that \(\lambda_1=\lambda_1(\Gamma_1)\) is the largest real root of \(P_{Q_1}(x)\). Since \(\lambda_1>n-\frac{5}{2}>n-3\), we get \(P_3(x)<0\) for all \(x\ge\lambda_1\). Since \(P_{Q_1}(x)\ge 0\) for \(x\ge \lambda_1\), it follows that \(P_{Q_5}(x)=x^2P_{Q_1}(x)-P_3(x)>0\) for \(x\ge \lambda_1\). Thus \(P_{Q_5}\) has no real zero in \([\lambda_1,\infty)\). As all other eigenvalues of \(Q_5\) are strictly less than \(\lambda_1\), we obtain \(\lambda_1(S)=\lambda_1(Q_5)<\lambda_1(\Gamma_1)\), a contradiction.
			
			\noindent \textbf{Case 2.} \(|\{r, s\} \cap \{a, b\}| = 1\).
			
			Without loss of generality, assume the missing edge is \(ar\). With respect to the partition
			\[
			\{w\}, \quad \{a\}, \quad \{b\}, \quad \{v_1\}, \quad \{v_2\}, \quad \{r\}, \quad R, \qquad |R| = n-6,
			\]
			the quotient matrix is
			\[
			Q_6 = \begin{bmatrix}
				0 & 1 & 1 & 1 & 0 & 0 & 0 \\
				1 & 0 & 1 & 1 & 1 & 0 & n-6 \\
				1 & 1 & 0 & 1 & 1 & 1 & n-6 \\
				1 & 1 & 1 & 0 & -1 & 1 & n-6 \\
				0 & 1 & 1 & -1 & 0 & 1 & n-6 \\
				0 & 0 & 1 & 1 & 1 & 0 & n-6 \\
				0 & 1 & 1 & 1 & 1 & 1 & n-7
			\end{bmatrix}.
			\]
			
			\noindent Its characteristic polynomial is
			\begin{align*}
				P_{Q_6}(x) &= x^7 + (7-n)x^6 + (18-5n)x^5 - (2n+8)x^4 \\
				&\quad + (17n-90)x^3 + (14n-52)x^2 + (55-8n)x + 13 - 3n.
			\end{align*}
			
			\noindent By construction, every vertex in a displayed cell has the same signed row sum toward each other cell. Hence the displayed partition is equitable, and Lemma~\ref{lemma:1} applies. By Lemma~\ref{lemma:1}, all eigenvalues of \(Q_6\) belong to the spectrum of \(A(S)\), and by subtracting a suitable scalar multiple of the all-ones block from each subblock of \(A(S)\), we derive the simplified block matrix
			\[
			A_6=
			\begin{bmatrix}
				0 & 0 & 0 & 0 & 0 & 0 & 0\\
				0 & 0 & 0 & 0 & 0 & 0 & \bm{0}\\
				0 & 0 & 0 & 0 & 0 & 0 & \bm{0}\\
				0 & 0 & 0 & 0 & 0 & 0 & \bm{0}\\
				0 & 0 & 0 & 0 & 0 & 0 & \bm{0}\\
				0 & 0 & 0 & 0 & 0 & 0 & \bm{0}\\
				0 & \bm{0} & \bm{0} & \bm{0} & \bm{0} & \bm{0} & -I_{n-6}
			\end{bmatrix}.
			\]
			Since \(\lambda_1(S)>0\), \(\operatorname{Spec}(A_6)=\{-1^{[n-6]},0^{[6]}\}\) and \(\lambda_1(S)=\lambda_1(Q_6)\).
			
			\noindent Recall that
			\[
			P_{Q_1}(x)=(x-1)\bigl(x^3+(5-n)x^2+(8-3n)x+2n-8\bigr).
			\]
			Define
			\[
			\begin{aligned}
				P_4(x)&=x^3P_{Q_1}(x)-P_{Q_6}(x)\\
				&= -3x^6 + (3n-15)x^5 + (7n-8)x^4 + (98-19n)x^3 \\
				&\quad + (52-14n)x^2 + (8n-55)x + 3n - 13.
			\end{aligned}
			\]
			Substituting \(x=n-\frac{5}{2}\) into \(P_4(x)\) and writing \(\eta=n-9\ge 0\), we obtain
			\[
			P_4\!\left(n-\frac{5}{2}\right)
			=-\frac{\eta^5}{2}-\frac{103\eta^4}{4}-\frac{1687\eta^3}{4}
			-\frac{24357\eta^2}{8}-\frac{319329\eta}{32}-\frac{761151}{64}<0.
			\]
			
			\noindent The derivative of \(P_4(x)\) is
			\[
			\begin{aligned}
				P_4'(x)&=-18x^5+(15n-75)x^4+(28n-32)x^3
				+(294-57n)x^2+(104-28n)x+(8n-55)\\
				&=\bigl(15x^4+28x^3-57x^2-28x+8\bigr)n
				+\bigl(-18x^5-75x^4-32x^3+294x^2+104x-55\bigr).
			\end{aligned}
			\]
			For \(x\ge 5\), the coefficient of \(n\) satisfies \(15x^4+28x^3-57x^2-28x+8>0\). Hence \(P_4'(x)\) is strictly increasing in \(n\). Since \(x\ge n-3\), i.e., \(n\le x+3\),
			\[
			\begin{aligned}
				P_4'(x)&\le P_4'(x)\big|_{n=x+3}\\
				&=-3x^5-2x^4-5x^3+95x^2+28x-31.
			\end{aligned}
			\]
			Define \(D(x)=-3x^5-2x^4-5x^3+95x^2+28x-31\). Writing \(t=x-5\ge 0\), we have
			\[
			D(5+t)=-3t^5-77t^4-795t^3-4030t^2-9772t-8766<0,
			\]
			so \(D(x)<0\) for all \(x\ge 5\). Consequently, \(D(x)\le D(5)<0\) for all \(x\ge n-3\ge 5\). This means that
			\[
			P_4'(x)\le D(x)<0.
			\]
			Thus \(P_4(x)\) is strictly decreasing for all \(x\ge n-3\). Recall that \(\lambda_1=\lambda_1(\Gamma_1)\) is the largest real root of \(P_{Q_1}(x)\). Since \(\lambda_1>n-\frac{5}{2}>n-3\), we get \(P_4(x)<0\) for all \(x\ge\lambda_1\). Since \(P_{Q_1}(x)\ge 0\) for \(x\ge \lambda_1\), it follows that \(P_{Q_6}(x)=x^3P_{Q_1}(x)-P_4(x)>0\) for \(x\ge \lambda_1\). Thus \(P_{Q_6}\) has no real zero in \([\lambda_1,\infty)\). As all other eigenvalues of \(Q_6\) are strictly less than \(\lambda_1\), we obtain \(\lambda_1(S)=\lambda_1(Q_6)<\lambda_1(\Gamma_1)\), a contradiction. Therefore, \(N_G(w) = \{v_1, v_2\}\).
		
		\begin{claim}
			\(\Gamma' - w = (K_{n-1}, \sigma)\).
		\end{claim}
		
		\begin{proof}
			Otherwise, assume that \(\Gamma' - w \neq (K_{n-1}, \sigma)\). Then \(\Gamma' - w\) contains at least two vertices that are non-adjacent. Without loss of generality, assume that \(v_i \nsim v_j\). Then we can construct a new signed graph \(\Gamma^*\) from \(\Gamma'\) by adding a positive edge \(v_i v_j\). Clearly, \(\Gamma^*\) is also unbalanced signed graph contains no negative Hamiltonian cycle and 
			\begin{align*}
				\lambda_1(A(\Gamma^*))-\lambda_1(A(\Gamma'))
				&\ge X^{T}\bigl(A(\Gamma^*)-A(\Gamma')\bigr)X \\
				&= 2x_i x_j \ge 0.
			\end{align*}
			If \(\lambda_1(A(\Gamma^*))=\lambda_1(A(\Gamma'))\), then \(X\) is also an eigenvector of \(A(\Gamma^*)\) corresponding to \(\lambda_1(A(\Gamma^*))\). By \hyperref[lemma:6i]{(i)} of Lemma~\ref{lemma:6}, we obtain that \(x_i=x_j=0\), which contradicts Claim ~\ref{claim:1}. Thus, $\lambda_1(\Gamma^*) > \lambda_1(\Gamma')$, a contradiction. 
		\end{proof}
		Therefore, \(\Gamma' \cong \Gamma_1\), which means that \(\Gamma\) is switching isomorphic to \(\Gamma_1\). This completes the proof.
		
	\end{proof}
	
	\section{Proof of Theorem \ref{thm:edge}}\label{3}
	
	In order to prove Theorem~\ref{thm:edge}, we first present some key lemmas and a useful remark that will be used in the subsequent arguments. Let $\Gamma^{s,t}$ (see Fig.~\ref{fig:signed_graph}) $(s+t = n-2)$ be the signed graph obtained from an all-positive clique $(K_{n-2}, +)$ with $V(K_{n-2}) = \{u_1, \dots, u_s, v_1, \dots, v_t\}$ ($s, t \ge 1$) and two isolated vertices $u$ and $v$ by adding the negative edge $uv$ and the positive edges in $\{uu_1, \dots, uu_s\}$ and $\{vv_1, \dots, vv_t\}$. 
	
	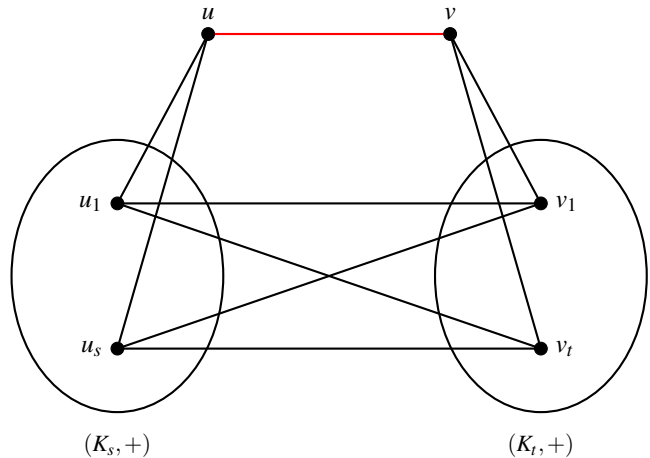
\begin{figure}[ht]
		\centering
		\begin{tikzpicture}[
			scale=0.8,          % 整体缩小
			transform shape,
			% 【修改2】全局设置顶点标签的字体大小
			every label/.style={font=\large},
			% 节点样式
			vertex/.style={circle, draw, fill=black, minimum size=6pt, inner sep=0pt},
			bigellipse/.style={ellipse, draw, thick, minimum width=3.5cm, minimum height=4.5cm},
			solidedge/.style={draw, thick},
			negedge/.style={draw=red, thick}
			]
			
			% 【修改1】拉大椭圆距离：x坐标从 -2.5/2.5 改为 -3.5/3.5
			\node[bigellipse] (Ks) at (-3.5, 0) {};
			\node[bigellipse] (Kt) at ( 3.5, 0) {};
			
			% 同步调整底部标签的x坐标
			\node at (-3.5, -2.8) {\((K_s, +)\)};
			\node at ( 3.5, -2.8) {\((K_t, +)\)};
			
			% 绘制所有的边
			% u 和 v 之间的红色负边（顶部顶点略向外扩以保持连线美观）
			\draw[negedge] (-2.0, 4.0) -- ( 2.0, 4.0);
			
			% u 到 u1, us 的边（x坐标同步调整）
			\draw[solidedge] (-2.0, 4.0) -- (-3.5, 1.2);
			\draw[solidedge] (-2.0, 4.0) -- (-3.5, -1.2);
			
			% v 到 v1, vs 的边
			\draw[solidedge] ( 2.0, 4.0) -- ( 3.5, 1.2);
			\draw[solidedge] ( 2.0, 4.0) -- ( 3.5, -1.2);
			
			% u1 到 v1, vs 的边
			\draw[solidedge] (-3.5, 1.2) -- ( 3.5, 1.2);
			\draw[solidedge] (-3.5, 1.2) -- ( 3.5, -1.2);
			
			% us 到 v1, vs 的边
			\draw[solidedge] (-3.5, -1.2) -- ( 3.5, 1.2);
			\draw[solidedge] (-3.5, -1.2) -- ( 3.5, -1.2);
			
			% 绘制所有的顶点
			\node[vertex, label=above:\(u\)] (u) at (-2.0, 4.0) {};
			\node[vertex, label=above:\(v\)] (v) at ( 2.0, 4.0) {};
			
			\node[vertex, label=left:\(u_1\)] (u1) at (-3.5, 1.2) {};
			\node[vertex, label=left:\(u_s\)] (us) at (-3.5, -1.2) {};
			
			\node[vertex, label=right:\(v_1\)] (v1) at ( 3.5, 1.2) {};
			\node[vertex, label=right:\(v_t\)] (vt) at ( 3.5, -1.2) {};
			
		\end{tikzpicture}
		\caption{The signed graph \(\Gamma^{s,t}\) (red lines indicate negative edges).}
		\label{fig:signed_graph}
	\end{figure}
	
	\begin{lemma}[\cite{Wang2024}]\label{lem:1}
		Let $\Gamma = (G, \sigma)$ be a connected unbalanced signed graph of order $n$. If $\Gamma$ is $C_3^-$-free, then
		\[
		e(\Gamma) \le \frac{n(n-1)}{2} - (n-2),
		\]
		with equality holding if and only if $\Gamma \sim \Gamma^{s,t}$, where $s+t = n-2$ and $s, t \ge 1$.
	\end{lemma}
	
	\begin{lemma}[\cite{Ore1963}]\label{lem:2}
		Let $G$ be a graph of order $n$. If 
		\[
		e(G) \ge \binom{n-1}{2} + 3,
		\]
		then $G$ is Hamilton-connected.
	\end{lemma}

	\begin{remark}\label{rem:one}
		An unbalanced signed graph has frustration index $1$ if and only if it is switching equivalent to a signed graph with exactly one negative edge. 
	\end{remark}
	\begin{proof}
		The sufficiency is obvious. Next, we prove the  necessity. Let $\Gamma = (G, \sigma)$ be a signed graph with frustration index $1$. Then there exists an edge $e$ such that $\Gamma - e$ is balanced. 
		Apply a sequence of switching operations to $\Gamma - e$ so that all its edges become positive. 
		Applying the same switching operations to the original graph $\Gamma$, we make every edge positive except possibly $e$. If $\sigma(e)\neq-1$ under this switching, then $\Gamma$ is balanced,  contradicting the assumption that $\Gamma$ is unbalanced. Thus, after these switching operations, $\Gamma$ is switching equivalent to a signed graph with exactly one negative edge and all other edges positive.
	\end{proof}

	\begin{proof}[\textbf{Proof of Theorem~\ref{thm:edge}}]
		Let $\Gamma$ be the extremal signed graph with $n$ vertices ($n \geqslant 7$) attaining the maximum number of edges. Then $\Gamma$ is indeed connected. Otherwise, for vertices $u$ and $v$ in different components, $\Gamma + uv$ is unbalanced, contains no negative Hamiltonian cycle, and has more edges than $\Gamma$, contradicting extremality. By Remark~\ref{rem:one}, let $\Gamma' = (G, \sigma')$ be a signed graph with exactly one negative edge that is switching equivalent to $\Gamma$, and $e(\Gamma)=e(\Gamma')$. Then $\Gamma'$ must contain at least one negative cycle, and we assume that the length of a shortest negative cycle of $\Gamma'$ is $\ell$. Note that $\Gamma_1$ is unbalanced and contains no negative Hamiltonian cycle, so
		\[
		e(\Gamma') \ge e(\Gamma_1) = \frac{n(n-1)}{2} - (n-3).\tag{1}\label{eq:edge-bound}
		\]
		If $\ell \ge 4$, then $\Gamma'$ is $C_3^-$-free. By Lemma~\ref{lem:1}, 
		\[
		e(\Gamma') \le \frac{n(n-1)}{2} - (n-2) < \frac{n(n-1)}{2} - (n-3),
		\]
		which contradicts the inequality \eqref{eq:edge-bound}. Hence, $\ell = 3$ and we assume that $uvwu$ is a negative triangle of $\Gamma'$ with $\sigma(uv) = -1$ and $\sigma(uw) = \sigma(vw) = +1$.
		Suppose $e(\Gamma') = \frac{1}{2}(n^2 - n) - q$. By the lower bound of the inequality \eqref{eq:edge-bound}, we have $q \le n-3$.
		
		\begin{claim}\label{cla:1}
			$q = n-3$.
		\end{claim}
		\begin{proof}
			Otherwise, $q\leq n-4$. Then $e(\Gamma') \ge \frac{n(n-1)}{2} - (n-4)$ = $\frac{(n-1)(n-2)}{2}+3$. Thus, the underlying graph $G$ is Hamilton-connected by Lemma~\ref{lem:2}. And then, there must exist a Hamiltonian path between $u$ and $v$, which together with the negative edge $uv$ forms a negative Hamiltonian cycle, a contradiction. Hence, $q = n-3$.
		\end{proof}
		
		Claim~\ref{cla:1} shows that $e(\Gamma')  = \frac{n(n-1)}{2} - (n-3).$	
		
		\begin{claim}\label{cla:2}
			$\delta(\Gamma') = 2$.
		\end{claim}
		
		\begin{proof}
			Suppose for contradiction that this is not the case. If $\delta(\Gamma') \ge 3$, note that  
			$e( K_3 \vee (K_{n-5} \cup I_2))= \frac{(n-2)(n-3)}{2}+6$ and for $n \ge 7 $, $e(\Gamma')$ is strictly larger than the edge numbers of these exceptional graphs. The relevant edge counts are listed in the table below. By Lemma~\ref{lem:ZhouWang}, the underlying graph $G$ is Hamilton-connected, which implies the existence of a negative Hamiltonian cycle of $\Gamma'$, a contradiction. If $\delta(\Gamma') = 1$, then $n-2\leq q $, which also contradicts Claim~\ref{cla:1}. Thus $\delta(\Gamma') = 2$. 
		\end{proof}
		
		\begin{table}[htbp]
			\centering
			\caption{Edge counts for the exceptional graphs.}
			\label{tab:edge-counts}
			\begin{tabular}{ccc}
				\toprule
				\(n\) & Maximum edges of exceptional graphs & Edges in the equality case of this paper \\
				\midrule
				8  & 22 & 23 \\
				9  & 27 & 30 \\
				10 & 35 & 38 \\
				12 & 51 & 57 \\
				\bottomrule
			\end{tabular}
		\end{table}
		
		 Let $u'$ be a vertex of degree~$2$ with \(N_{\Gamma'}(u') = \{a, b\}\).
		
		\begin{claim}\label{cla:3}
			Let $uv$ be the unique negative edge of $\Gamma'$, and let $u'$ be any vertex of degree $2$ in $\Gamma'$. Then $\{a,b\}=\{u,v\}.$
		
		\end{claim}
		\begin{proof}
			The vertex $u'$ is incident with exactly $(n-1)-2=n-3$ non-edges. By Claim~\ref{cla:1}, the total number of non-edges of $\Gamma'$ is $q=n-3$. Hence all non-edges are incident with $u'$, which yields $G-u'=K_{n-1}$. Consequently, every vertex of $\Gamma'-u'$ has degree at least $n-2>2$, so $u'$ is the unique vertex of degree $2$ in $\Gamma'$. We first rule out \(u'=u\) and \(u'=v\). Suppose that \(u'=u\). After interchanging \(a\) and \(b\) if necessary, we may assume that \(b=v\). Then $N_{\Gamma'}(u)=\{v,a\}$. Since $\Gamma'-u=(K_{n-1},+)$, there exists a Hamiltonian path from $v$ to $a$ in $\Gamma'-u$. This path consists entirely of positive edges, and together with $uv$ and $ua$ it forms a negative Hamiltonian cycle containing only the negative edge $uv$, a contradiction. The case $u'=v$ is analogous. Therefore $u'\notin\{u,v\}$. Finally, suppose for contradiction that $\{a,b\}\neq\{u,v\}$. Then $ab$ and $uv$ are two distinct edges of $G-u'=K_{n-1}$. In a complete graph, any two distinct edges can be simultaneously contained in a Hamiltonian cycle: if the two edges are disjoint, we may arrange them as two consecutive segments; if they share a common endpoint, we may arrange them as a path of length $2$ with that endpoint as the middle vertex. Thus $\Gamma'-u'$ contains a Hamiltonian cycle containing both $ab$ and $uv$. We replace the edge \(ab\) in this Hamiltonian cycle by the path \(a-u'-b\). Since \(u'a\) and \(u'b\) are positive and \(ab \neq uv\), the resulting Hamiltonian cycle still contains the unique negative edge \(uv\), and hence is a negative Hamiltonian cycle, a contradiction. Therefore $\{a,b\}=\{u,v\}$. Since $G-u'=K_{n-1}$, $N_{\Gamma'}(u')=\{u,v\}$, and $uv$ is the unique negative edge of $\Gamma'$, we conclude 
			that \(\Gamma' \cong \Gamma_1\), which means that \(\Gamma\) is switching isomorphic to \(\Gamma_1\).
		\end{proof}
		
   This completes the proof of Theorem~\ref{thm:edge}. \qed
\end{proof}
	
	\section*{Acknowledgements}
	The author thanks Professor Dan Li for helpful discussions.

\end{document}